\documentclass{amsart}
\usepackage{amssymb,mathrsfs,amsmath}
\input xy
\xyoption{all}

\usepackage[all]{xy}
\usepackage{color}
\newtheorem{thm}{Theorem}

\newtheorem{lemma}{Lemma}[section]
\newtheorem{prop}{Proposition}[section]

\newtheorem{defn}{Definition}

\newtheorem{remark}{Remark}[section]
\newtheorem{example}{Example}[section]
\makeatletter
\DeclareSymbolFont{script}{U}{eus}{m}{n}
\DeclareMathSymbol{\Wedge}{0}{script}{"5E}
\DeclareSymbolFont{rmslops}{OT1}{cmr}{m}{sl}
\DeclareSymbolFontAlphabet{\mathrmsl}{rmslops}
\def\operator@font{\mathgroup\symrmslops}
\makeatother
\newcommand{\acknowledge}{\subsection*{Acknowledgements}}
\newcommand{\calsyma}[2]{\newcommand{#1}{{\mathcal{#2}}}}
\newcommand{\calsymb}[2]{\newcommand{#1}{{\mathscr{#2}}}}
\newcommand{\bbsymb}[2]{\newcommand{#1}{{\mathbb{#2}}}}
\newcommand{\liealg}[2]{\newcommand{#1}{{\mathfrak{#2}}}}
\bbsymb\C{C} \bbsymb\HQ{H}\bbsymb\N{N} \bbsymb\Q{Q}
\bbsymb\R{R} \bbsymb\V{V} \bbsymb\W{W} \bbsymb\Z{Z}
\calsymb\cA{A} \calsymb\cB{B} \calsymb\cC{C} \calsymb\cD{D} \calsymb\cE{E}
\calsymb\cF{F} \calsymb\cG{G} \calsymb\cH{H} \calsymb\cI{I} \calsymb\cJ{J}
\calsymb\cK{K} \calsymb\cL{L} \calsymb\cM{M} \calsymb\cN{N} \calsymb\cO{O}
\calsymb\cP{P} \calsyma\cQ{Q} \calsymb\cR{R} \calsymb\cS{S} \calsymb\cT{T}
\calsymb\cU{U} \calsymb\cV{V} \calsymb\cW{W} \calsymb\cX{X} \calsymb\cY{Y}
\calsymb\cZ{Z}
\liealg\gl{gl}\liealg\sgl{sl}\liealg\symp{sp}
\liealg\g{g}\liealg\p{p}

\DeclareMathOperator{\Aff}{Aff}

\DeclareMathOperator{\spanu}{span}

\DeclareMathOperator{\id}{id}

\renewcommand{\d}{{\mathrmsl d}}

\newcommand{\hp}[1][^]{#1{1,0}}
\newcommand{\am}[1][^]{#1{0,1}}
\newcommand{\Vhp}{\cV\hp}
\newcommand{\Vam}{\cV\am}

\newcommand{\php}{\phi\hp[_]}
\newcommand{\phm}{\phi\am[_]}

\newcommand{\Lp}{\cL\hp[_]}
\newcommand{\Lm}{\cL\am[_]}

\newcommand{\sub}{\subseteq}
\newcommand{\tp}{\otimes}
\newcommand{\ds}{\oplus}

\newcommand{\ra}[1]{{\raise6pt\hbox{$#1$}}}

\newcommand{\mult}{^{\scriptscriptstyle\smash[b]{\times}}}
\begin{document}

\title{Quaternion-K\"ahler manifolds near maximal fixed points sets of $S^1$-symmetries}
\author{Aleksandra  Bor\'owka}
\address{Aleksandra W. Bor\'owka \\
Institute of Mathematics\\ Polish Academy of Sciences\\
ul. \'Sniadeckich 8\\
00-656 Warszawa\ Poland\\
aleksandra.borowka@uj.edu.pl}
\maketitle
\begin{abstract}
Using quaternionic Feix--Kaledin construction we provide a local classification of quaternion-K\"ahler metrics with a rotating 
$S^1$-symmetry with the fixed point set submanifold $S$ of maximal possible dimension. For any K\"ahler manifold $S$ equipped with a line bundle with a unitary connection of curvature proportional to the K\"ahler form we explicitly construct a holomorphic contact distribution on the twistor space obtained by the quaternionic Feix-Kaledin construction from this data. Conversely, we show that quaternion-K\"ahler metrics with a rotating $S^1$-symmetry induce on the fixed point set of maximal dimension a K\"ahler metric together with  a unitary connection  on a holomorphic line bundle with curvature proportional to the K\"ahler form and the two constructions are inverse to each other. Moreover, we study the case when $S$ is compact, showing that in this case the quaternion-K\"ahler geometry is determined by the K\"ahler metric on the fixed point set (of maximal possible dimension) and by the contact line bundle. Finally, we relate the results to the c-map construction showing that the family of quaternion-K\"ahler manifolds obtained from a fixed K\"ahler metric on $S$ by varying the line bundle and the hyperk\"ahler manifold obtained by hyperk\"ahler Feix--Kaledin construction form $S$ are related by hyperk\"ahler/quaternion-K\"ahler correspondence.
\end{abstract}
\section{Introduction}

A classification of hyperk\"ahler metrics with rotating $S^1$-symmetry near a fixed points submanifold $S$ of maximal dimension was provided by B. Feix \cite{Feix} and, independently,  by D. Kaledin \cite{Kal}. They have shown that locally such metrics are uniquely determined by the real-analytic K\"ahler  metric on $S$ and for any such metric the corresponding hyperk\"ahler  structure can be constructed on a neighbourhood of the zero section of the cotangent bundle of $S$. The twistorial construction of Feix have been generalized in \cite{BC}, where a classification of quaternionic structures with rotating $S^1$-symmetry near 
the fixed points submanifold $S$ of dimension $n$ was given. In this case $S$ is, so called, real-analytic c-projective manifold with c-projective curvature of type $(1,1)$ and the structure is constructed on a neighbourhood of the zero section of the tangent bundle twisted by some unitary line bundle. Although considering a c-projective structure (instead of a K\"ahler metric) on $S$ is interesting mainly from non-metric point of view, introducing a twist by a line bundle with a connection allows obtaining non-zero scalar curvature analogues of hyperk\"ahler metrics namely quaternion-K\"ahler metrics. In \cite{BC} it is shown that all quaternionic structures admitting an $S^1$-action in the neighbourhood of the fixed points set $S$ of maximal dimension arise locally in this way, provided that the action on $S$ is not triholomorphic, i.e., that the action is rotating complex structures from the quaternionic family. 

The aim of this paper is to use results from \cite{BC} to provide an analogous classification result for quaternion-K\"ahler metrics with rotating $S^1$-symmetry. We show that near the fixed points submanifold $S$ of dimension $n$ such metrics are uniquely determined by the real-analytic K\"ahler metric on $S$ and a holomorphic line bundle on $S$ equipped with a real-analytic unitary connection with curvature proportional to the K\"ahler form. On the twistor space, the manifold $S$ corresponds to Legendrian submanifolds $S^{1,0}\cong S$ and $S^{0,1}\cong\overline{S}$ and the square of the line bundle is the contact line bundle restricted to $S^{1,0}$. When $S$ is compact, then by uniqueness of hermitian-Einstein metrics we deduce that the metric locally near $S$ is determined by the K\"ahler structure on $S$ and the contact line bundle. Note, that in \cite{BC}  the case when $S$ is K\"ahler-Einstein and the line bundle is a root of the canonical bundle is already studied, namely it has been proved that then the constructed quaternionic manifold $M$ is quaternion-K\"ahler. The proof is non-explicit and involves studying hypercomplex cones (i.e., Swann bundles) over $M$ and Armstrong cones over $S$. In particular the c-projective Armstrong cone over $S$ carries a natural  K\"ahler metric if and only if $S$ is  K\"ahler-Einstein which could suggest that this condition is necessary. In this paper we show that this is not true: starting from any K\"ahler manifold with a holomorphic line bundle on $S$ with a real-analytic unitary connection such that its curvature is proportional to the K\"ahler form we explicitly construct a holomorphic contact distribution on the twistor space. 

For a fixed K\"ahler structure on $S$ by varying the line bundle we obtain a family of  quaternion-K\"ahler manifolds which for the  trivial line bundle degenerates to a unique hyperk\"ahler metrics. As already observed in \cite{BC} when $S$ is  K\"ahler-Einstein, this is related to hyperk\"ahler/quaternion-K\"ahler correspondence. In mathematical language this has been introduced by A. Haydys \cite{Hay}, where he uses a concept of the physicist's c-map and the twistorial description of this correspondence was given by N. Hitchin \cite{Hit2}. The manifolds appearing in the correspondence are equipped with a rotating $S^1$-symmetry and it is assumed that the action is locally free. We show that our approach extends this correspondence to neighbourhoods of fixed points sets of the $S^1$-action, provided that it is of maximal dimension.

The structure of the paper is as follows. In Section \ref{Pre} we provide necessary background including the outline of quaternionic Feix--Kaledin construction \cite{BC}. In Section \ref{nece} we give conditions on the data in quaternionic Feix--Kaledin construction which are  necessary to obtain as a result a quaternion-K\"ahler manifold with a rotation $S^1$-isometry and in Section \ref{dist} we prove that this continuous are sufficient. In Section \ref{gl} we discuss the case when $S$ is compact, proving in particular that the  quaternion-K\"ahler on a neighbourhood of $S$ is fully encoded in the K\"ahler metric on $S$ and the contact line bundle restricted to $S$ on the twistor space. We finish in Section \ref{ex} by discussing examples and relation with the  hyperk\"ahler/quaternion-K\"ahler correspondence.

 \section{Preliminaries}\label{Pre}
 \subsection{Complexification}\label{comp} We start by giving a short summary about complexifications. For more detailed description see \cite{Biel}, \cite{BC}, \cite{Le}. For any real-analytic $m$-dimensional manifold $S$ using holomorphic extensions of real analytic coordinates we can construct its complexification i.e., a complex manifold of complex dimension $m$ with a real structure whose fixed points set - the real submanifold - is isomorphic to $S$. Similarly we can use holomorphic extensions to extend real-analytic objects like bundles and connections from $S$ to its complexification. Note that locally near the real submanifold any two complexifications are equal. If $S$ is itself a complex manifold, then there is a convenient model of complexification namely the manifold $S^{\mathbb{C}}=S\times \overline{S}$, where $\overline{S}$ denote $S$ with the opposite complex structure, the real structure is the involution $(x,y)\mapsto (y,x)$ and its fixed points set is the diagonal. In this setting, if we consider complex local coordinates on $S$ and their complex conjugates $z_1,\overline{z}_1,\ldots, z_n,\overline{z}_n$, $m=2n$, then the complexification is achieved by replacing coordinates $\overline{z}_1,\ldots, \overline{z}_n$ by independent coordinates $\tilde{z}_1,\ldots,\tilde{z}_n$ and then the real submanifold is given by the equation $\tilde{z}_i=\overline{z}_i$.
The manifold $S^{\mathbb{C}}$ carries two transverse foliations  given by the product structure and the leaves are isomorphic to $S$ and $\overline{S}$ correspondingly. We denote the foliations as the $(1,0)$-foliation and the $(0,1)$-foliation correspondingly and the leave spaces by $S^{1,0}$ and $S^{0,1}$.
 
 For a holomorphic rank $k$ vector bundle $L$ on a complex manifold $S$ we will understand by its complexification $L^{\mathbb{C}}$ a holomorphic extension to a tubular neighbourhood of the diagonal in $S^{\mathbb{C}}$ of the bundle $L\otimes_{\mathbb{R}}\mathbb{C}$. $L^{\mathbb{C}}$ is a holomorphic bundle of rank $2k$ and the holomorphic structure on $L$ induces its decomposition into a sum of two holomorphic  bundles of rank $k$ $$L^{\mathbb{C}}=L^{1,0}\oplus L^{0,1},$$ such that $L^{1,0}$ is trivial along the leaves of the $(0,1)$-foliation and is the pull-back of $L$ from $S^{1,0}\cong S$ and $L^{0,1}$ is trivial along the leaves of the $(1,0)$-foliation and is the pull-back of $\overline{L}$ from $S^{0,1}\cong \overline{S}$. Any connection $\nabla$ compatible with a holomorphic structure on $L$ gives connections on $L^{1,0}$ and $L^{0,1}$ trivial in the directions where bundles are trivial and one can show that the connections are flat in the other direction exactly when the curvature  of $\nabla$ is of type $(1,1)$.
 \subsection{Twistor theory for quaternionic manifolds}\label{twist}
 
 Quaternionic manifold  $M$ is a smooth manifold of dimension $4n$, $n>1$ equipped with a rank $3$ subbundle $Q\subseteq End (TM)$ spanned locally by three anti-commuting almost complex structures satisfying quaternionic relations, together with a torsion-free connection $D$ with $DQ=0$ called a quaternionic connection. Inside $Q$ there is a $2$-sphere subbundle $\cQ$ of almost complex structures and the total space of this bundle, denoted $Z$, is called a twistor space. It turns out that $Z$ is equipped with an almost complex structure $\mathcal{I}$ and the existence of torsion-free quaternionic connection is exactly the integrability condition for $\mathcal{I}$ (see \cite{Sal}). Note that a torsion-free quaternionic connection is never unique (in fact if one such connection exists, there is an affine space of such connections modeled on the space of $1$-forms on $M$) but one can show that the definition of $\mathcal{I}$ does not depend on the choice of $D$. Therefore twistor spaces of quaternionic manifolds are complex manifolds of complex dimension $2n+1$ and they are equipped with the following objects:
 \begin{itemize}
 \item[(i)] family of projective lines with normal bundles $\mathbb{C}^{2n}\otimes\mathcal{O}(1)$ called twistor lines,
 \item[(ii)] a real structure $\tau$ without fixed points which  is invariant on some twistor lines, called real twistor lines.
 \end{itemize}
 A twistor space fully determines quaternionic geometry on $M$ and  it has been proved in  \cite{PP} that any complex  $(2n+1)$-manifold $Z$ with properties (i) and (ii) is locally (near some fixed real twistor line) a twistor space of some quaternionic manifold,  which as a smooth manifold is the moduli space of real twistor lines.

On a quaternionic manifold one can always define a $Q$-hermitian metric $g$, however, the Levi-Civita connection of such metric needs not be quaternionic (i.e., in general $D^g$ may not preserve the bundle $Q$). If the Levi-Civita connection  is quaternionic, then $(M, g, Q)$ is called quaternion-K\"ahler manifold. Quaternionic connections give $\tau$-invariant complex distributions  on $Z$ (but not necessarily holomorphic, see \cite{Alex}, Section 5.3 for characterization of holomorphic distributions).
\begin{defn}
A holomorphic contact manifold $Z$  is a $(2n+1)$-dimensional complex manifold equipped with a contact form $\theta\in\Gamma (T^*Z\otimes L)$ with values in some line bundle $L$ with the property
$$\theta\wedge(\d\theta)^n\neq 0.$$
A distribution is called contact if it is the kernel of a contact form.
\end{defn}
 If a quaternionic connection connection is the Levi-Civita connection of a  quaternion-K\"ahler metric then the distribution is contact and holomorphic and in particular $Z$ is a holomorphic contact manifold (\cite{Sal0}, Theorem 4.3). 

 Moreover,  as shown in \cite{LeBrun} Theorem 1.3, the following converse is also true. Let $Z$ be a complex manifold of complex dimension $2n+1$ satisfying the conditions (i), (ii) and
\begin{itemize}
\item[(iii)] Z is a holomorphic contact manifold with $\tau$-invariant contact distribution transversal to real twistor lines,
\end{itemize}
then $Z$ is locally a twistor spaces of a pseudo quaternion-K\"ahler manifold.

In this case, the line bundle $L$ is the vertical bundle of the twistor fibration (i.e., the bundle of tangent spaces to twistor lines, note that for general quaternionic manifold this bundle needs not to be holomorphic \cite{Kob}) and it is the $(n+1)$-st root of the anti-canonical bundle of $Z$.

 It is often convenient to keep in mind the complexified version of the twistor correspondence: the moduli space of all twistor lines is locally a $4n$-dimensional complex manifold $M^{\mathbb{C}}$ with a real structure induced by $\tau$ whose fixed points correspond to $\tau$-invariant twistor lines: $M^{\mathbb{C}}$ is a complexification of the quaternionic manifold $M$ (note that quaternionic manifolds are always real-analytic). 
 \begin{defn}\label{inc}
 The incidence space for the complexified twistor correspondence defined is the space
 $$\cF:=\{(z,u)\in Z\times M^{\mathbb{C}}:\ z\in u\}.$$
 \end{defn}
The space $\cF$ is a sphere bundle over $M^{\mathbb{C}}$ and it can be viewed as a complexification of the sphere bundle of quaternionic structures over $M$. Moreover, a complexification of a quaternionic connection $D$ on $M$ induces a distribution on $\cF$ which is the pull back from $Z$ to $\cF\subset Z\times M^{\mathbb{C}}$  of the distribution given by $D$ on $Z$.
 An important fact about complexifications of quaternionic manifolds is that locally the tangent space can be decomposed $TM^{\mathbb{C}}=\cE \otimes_{\mathbb{C}} \cH$, where  $\cE$ and $\cH$ are bundles associated to the standard complex representation of $Sp(n)$ and $Sp(1)$ respectively. Many properties of quaternion-K\"ahler manifolds can be proved by studying these representations. In particular the bundle $S^2\cE\oplus S^2\cH $ is naturally a subbundle of the bundle of $2$-forms but also a subbundle of $End(TM)$ and $S^2\cH $ is spanned by the almost complex structures compatible with the quaternionic structure.

For $n=1$ we define quaternionic manifolds as self-dual conformal $4$-manifolds and quaternion-K\"ahler as self-dual Einstein $4$-manifolds and in this setting the twistor correspondence extends to the dimension $n=1$.
\subsection{The quaternionic Feix--Kaledin construction} \label{Feix}
The quaternionic Feix--Kaledin \cite{BC} (qFK in short) construction is a generalization of the hypercomplex and hyperk\"ahler  twistorial constructions of Feix \cite{Feix}. Using real-analytic geometric data on a complex manifold $S$ it gives a quaternionic structure with a compatible $S^1$-action on a neighbourhood of the zero section of a twisted tangent bundle to $S$  and it has been proved that in this way we can obtain any quaternionic structure with a quaternionic $S^1$-action locally near the fixed point set provided that it is of maximal possible dimension and the action has no triholomorphic points.
In this section we will briefly summarize the construction from \cite{BC} as the main aim of this paper is to characterize when quaternionic manifolds obtained in this way are quaternion-K\"ahler. Note that we will present a simplified version of the construction by replacing a c-projective structure on $S$ by a fixed connection as only this case is relevant in the metric setting (the quaternion-K\"ahler metric restricts to a K\"ahler metric on $S$). We will also outline some ideas from the original approach of Feix which will be relevant here.

Let $S$ be a complex manifold equipped with a torsion-free complex connection $D$ with curvature of type $(1,1)$ (for example the Levi-Civita connection of a K\"ahler metric) and $\cL$ be a holomorphic line bundle with a connection $\nabla_{\cL}$ compatible with the holomorphic structure and with curvature of type $(1,1)$. As we have discussed in Section \ref{comp}, the complexification of the bundle $\cL$ decomposes into the sum of two line bundles $\cL^{1,0}$ and $\cL^{0,1}$ which are trivial along the leaves of the $(0,1)$ and the $(1,0)$-foliations respectively, and equipped with the connections flat along the leaves of the $(1,0)$ and the $(0,1)$-foliation respectively. Also the connection $D^c$ on $TS^{\mathbb{C}}$ is flat along  leaves of the $(1,0)$ and the $(0,1)$-foliations respectively. 
\begin{defn}\label{o}
Let $S$ be a smooth complex manifold of complex dimension $n$. The bundle $\mathcal{O}(1)$ over an open subset of $S$ is defined as a local  $(n+1)$-st root of $\Lambda^nT^{1,0}S$, where $T^{1,0}S$ denote the holomorphic tangent bundle. Note that if $S=\mathbb{CP}^n$ then the bundle $\mathcal{O}(1)$ is defined globally and it is equal to the dual to the tautological bundle.  
\end{defn}
 Let us denote by $\Lp:=\cL^{1,0}\otimes\mathcal{O}(1)$ and $\Lm:=\cL^{0,1}\otimes\mathcal{O}(1)$ and by $\nabla$ the connections on $\Lp$ and $\Lm$ induced by $\nabla_{\cL}$ and $D$. 

Over each leaf of the foliations one defines $(n+1)$-dimensional complex vector space: in \cite{Feix} this is the vector space of `affine functions' along the leaf, i.e. these which along the leaf satisfy $D^cdf=0$, while in \cite{BC} it is the space of `affine sections' of $\Lp$ and $\Lm$ respectively defined as projections of parallel sections of a flat connection on the $1$-jet bundle of $\Lp$ and $\Lm$ along the leaf. In our simplified case, the sections are tensor products of affine functions with parallel sections for $\nabla$ on $\Lp$ and $\Lm$ respectively. In this way, we define two rank $n+1$ holomorphic bundles over the leave spaces of the foliations: $\Aff(\Lp)$ over $S^{0,1}$ and $\Aff(\Lm)$ over $S^{1,0}$ and we further define $$\Vhp:=[\Aff(\Lp)]^*\otimes \Lm,\ \ \ \ \ \ \ \  \Vam:=[\Aff(\Lm)]^*\otimes \Lp.$$
The  evaluation induces maps $\phi_{1,0}$ from the total space of the bundle $(\Lp)^*\otimes\Lm\rightarrow S^{\mathbb{C}}$ to $\Vhp$ and $\phi_{0,1}$ from $(\Lm)^*\otimes\Lp\rightarrow S^{\mathbb{C}}$ to $\Vam$ which behave like a blow-down: they are biholomorphisms on the complements of the zero sections and on the zero sections they contract leaves of one of the foliations; the images are cone subbundles of $\Vhp$ and $\Vam$. 

The projective bundle $\mathbb{P}([(\cL\otimes\mathcal{O}(1))^*]^{\mathbb{C}})=\mathbb{P}([\Lp]^*\oplus[\Lm]^*)$ has canonical $0$ and $\infty$ sections and the corresponding affine bundles (obtained by removing one of the sections) are  $(\Lp)^*\otimes\Lm$ and  $(\Lm)^*\otimes\Lp$. This defines a gluing of $\Vhp$ and $\Vam$ to a possibly non-Hausdorff complex manifold. However, we can choose open tubular neighbourhoods of the zero sections of $\Vhp$ and $\Vam$ containing the images of $\phi_{1,0}$ and $\phi_{0,1}$ such that the gluing is a Hausdorff manifold $Z$. 

It turns out that $Z$ satisfies the properties of quaternionic twistor space: the real structure is induced by the real structure $\tau$ given by the complexification and the images of fibres of $\mathbb{P}([(\cL\otimes\mathcal{O}(1))^*]^{\mathbb{C}})$ are twistor lines. The twistor lines obtained in this way are the twistor lines corresponding to $S^{\mathbb{C}}\subset M^{\mathbb{C}}$, where $M$ is the quaternionic manifold obtained from $Z$. To simplify the notation we make the following definition.
\begin{defn}\label{canonical}
The twistor lines which are images of fibres of $\mathbb{P}([(\cL\otimes\mathcal{O}(1))^*]^{\mathbb{C}})$ are called canonical twistor lines and the fibres over the real submanifold $S\subset S^{\mathbb{C}}$ are called canonical real twistor lines.  The canonical twistor lines correspond to points in a neighbourhood of $S$ in the submanifold $S^{\mathbb{C}}\subset M^{\mathbb{C}}$.
\end{defn}
 The construction can be summarized by the following diagram.
\begin{equation} \label{dig1}
\xymatrix@R=1.1cm@C=.9cm{
& \Lp^*\tp\Lm\ar@{^{(}->}[dl] \ar[r]^-{\phi_{1,0}} & \Vhp \ar@{~>}[dr] &\\
\mathbb{P}(\Lp^{*}\ds\Lm^{*})\ar@{.>}[rrr]^{\phi}&&&Z\\
&\Lp\tp\Lm^*\ar@{_{(}->}[ul] \ar[r]^-{\phi_{0,1}}&\Vam\ar@{~>}[ur]&
}
\end{equation}

The scalar multiplication (and its inverse) in the fibres induces a holomorphic $\tau$-invariant $\mathbb{C}^*$-action on $Z$ fixing the zero sections of $\Vhp$ and $\Vam$, tangent to canonical twistor lines and transversal to twistor lines near the canonical real twistor lines. Therefore, it corresponds to a quaternionic $S^1$-action on $M$ such that $S\subset M$ is the fixed points set.
\begin{defn}
A fixed point of a quaternionic $S^1$-action on a quaternionic manifold is called triholomorphic if the induced action on the twistor space fixes all points on the corresponding twistor line.
\end{defn}
The points on the fixed point set $S\subset M$ of the action obtained by qFK are not triholomorphic.  In \cite{BC} it is shown that the twistor space of any quaternionic manifold with such an $S^1$-action can be obtained locally in this way: the fixed point set corresponds on the twistor space $Z$ to a submanifold of complex dimension $n$ with two components on which the corresponding $\mathbb{C}^*$-action has constant single weight equal to $1$ and one can deduce that $Z$ is of the form as in the quaternionic Feix--Kaledin construction.

Note that if $M$ is a quaternion-K\"ahler manifold, then any isometry is also a symmetry of quaternionic structure \cite{Alexsym} Section 5 (but the converse is not true). Hence we will use terms $S^1$-symmetry and $S^1$-isometry in this setting interchangeably.

\subsection{Quaternion-K\"ahler moment map} \label{moment}
For any $S^1$-symmetry (i.e., isometry) on a quaternion-K\"ahler manifold $M$ one can define its moment map (\cite{Gal})  in the following way. Let $\omega_i(\cdot,\cdot):=g(I_i\cdot,\cdot)$, where $I_1,I_2,I_3$ is a local quaternionic base. Then the $4$-form $\Omega:=(\omega_1)^2+ (\omega_2)^2+(\omega_3)^2$ is defined globally, is non-degenerate and does not depend on the choice of the local quaternionic bases. The moment map of the action is a section of $S^2\cH$ (where $TM^{\mathbb{C}}=\cE\otimes_{\mathbb{C}} \cH$ and  $S^2\cH$ is a subbundle of $2$-forms, see Section \ref{twist}) satisfying 
$$\d \mu=i_X\Omega,$$
where $X$ is the Killing field of the action.
The moment map is unique and on the twistor space it corresponds to the section $\theta(\hat{X})$ (see \cite{Hit2} Section 4.3), where $\hat{X}$ is the vector field on $Z$ corresponding to $X$ and $\theta$ is the contact form. In \cite{Bat} Proposition 3.3, it is shown that away from $\mu^{-1}(0)$, the moment map induces a distinguished (up to a sign) integrable complex structure $\pm I$ invariant under the $S^1$-action. From the twistorial point of view, the complex structure corresponds to one of the two components of the zero set of the twistorial moment map $\theta(\hat{X})$: along real twistor lines the bundle $L$ (in which $\theta$ takes the values) is the bundle $\mathcal{O}(2)$, hence $\theta(\hat{X})$ along  any real  twistor line either vanishes or has exactly two zeros. The first possibility is excluded on the complement of $\mu^{-1}(0)$, hence locally we have two components of the zero set of $\theta(\hat{X})$. 

If $J$ is any almost complex structure from the quaternionic structure anti-commuting with $I$ (possibly defined only locally) then we have that $\mathcal{L}_XJ=K$, where $K=IJ$. 
 \section{The necessary condition for the existence of a quaternion-K\"ahler metric}\label{nece}
In the literature one can find two different definitions of  totally complex submanifolds of quaternion-K\"ahler manifolds. One of them is connected purely with the underlying quaternionic structure (see \cite{BC}), while the other gives also a condition for the metric (see \cite{Alex2}). Recall that $\cQ$ is the $2$-sphere subbundle  of $Q$  (which is the bundle of almost complex structures forming the quaternionic structure). 
\begin{defn}
 A submanifold $S$ of $(M,Q)$ is totally complex if there exists
 a section $I$ of $\cQ|_S$ (in particular $I^2=-\id$) such that the following two conditions are satisfied:
\begin{itemize}
\item[(i)] $I(TS)\sub TS$ (so that $I$ is an almost complex structure on $S$);
\item[(ii)] for all $J$ anti-commuting with $I$, $J(TS)\cap TS = 0$.
\end{itemize}

 A submanifold $S$ of $(M,Q)$ is metrically totally complex if it satisfies the  conditions $(i), (ii)$ and the following condition $(iii)$:
\begin{itemize}
\item[(iii)] $J(TS)$ is $g$-orthogonal to $TS$ for any $J$ anti-commuting with $I$.
\end{itemize}

\end{defn}
Note that as it is shown in \cite{Alex1} Theorem 5.2, the above conditions imply that $J$ is integrable hence $S$ is a complex submanifold of $M$.

Let $Z$ be a twistor space of a quaternion-K\"ahler manifold $M$ with an $S^1$-action which is an isometry, has no triholomorphic points and its fixed point set component $S$ is of the maximal possible dimension (i.e. of complex dimension $n$).  Then locally near $S$, the manifold $M$ can be obtained by the quaternionic Feix--Kaledin construction (see Section \ref{Feix}) and the moment map has no zeros as the  action is rotating the twistor lines corresponding to $S$. We restrict $(Z,M,S)$ such that they globally arise by the construction. As described is Section \ref{Feix}, the twistor space is in this case a gluing of tubular neighbourhoods of zero sections of two holomorphic rank $(n+1)$ vector bundles $\Vhp$ and $\Vam$ and the zero sections of the bundles correspond to $(S,I)$ and $(S,-I)$, where $I$ is the complex structure on $S$ fixed by the action. The bundle structure on  $\Vhp$ and $\Vam$ is given by the holomorphic $\mathbb{C}^{\mult}$-action on $Z$ corresponding to the $S^1$-action. 
 As $(M,g)$ is quaternion-K\"ahler, the  distribution induced by the Levi-Civita connection on $Z$ is holomorphic and contact and we denote by $D^g$ the quaternionic-K\"ahler connection on $M$ and by $\theta$ the contact form. We abuse the notation and denote the holomorphic contact distribution on $Z$ (i.e., the kernel of $\theta$) also by $D^g$.

As explained in Section \ref{moment}, any quaternion-K\"ahler manifold with a compatible circle action on the complement of the zero set of the moment map admits a distinguished (up to a sign) $S^1$-invariant compatible integrable complex structure $I$, which is defined by the zero set of the twistorial moment map. In  our case this complex structure is an extension of the complex structure on $S$ and it follows from \cite{Bat} Proposition 3.5 that $S$ (as contained the fixed point set) is K\"ahler and hence by  \cite{Alex2} Theorem 1.12 metrically totally complex. The complex structure $I$ on $Z$  is a holomorphic smooth divisor $\cD^{1,0}$ which contains the zero section of the bundle $\Vhp$. 
\begin{lemma}\label{div}
Let $M$ be a quaternion-K\"ahler manifold obtained by qFK construction such that the natural $S^1$-action is an isometry. Then the divisor  $\cD^{1,0}$ on $Z$ corresponding to the distinguished (up to a sign) $S^1$-invariant complex structure $I$ defined by the moment map is a rank $n$ subbundle of the vector bundle $\Vhp$.
\end{lemma}

\begin{proof}
Recall that the $S^1$-action on $M$ gives a $\mathbb{C}^{\mult}$-action on $Z$ which on $\Vhp$  is the scalar multiplication in the fibres. The divisor $\cD^{1,0}$ represent the unique (up to a sign) complex structure $I$ which is $S^1$-invariant, therefore $\cD^{1,0}$ is $\mathbb{C}^{\mult}$-invariant. Hence, as $\cD^{1,0}$ is smooth, it is a holomorphic rank $n$ subbundle of $\Vhp$ transversal to canonical real twistor lines (see Definition \ref{canonical}). 
\end{proof}
As a result we have that  $(M,I)\cong  \cD^{1,0}$ and this isomorphism is $S^1$-eqivariant. 

 Since $Z$ is a holomorphic contact $(2n+1)$-manifold, through any point of $Z$ locally there are Legendrian $n$-dimensional holomorphic submanifolds tangent to the contact distribution. Let  $u_{x,\overline{x}}$ be the real twistor line through the zero of $\cD^{1,0}_{\overline{x}}$ for $\overline{x}\in S^{0,1}$. The twistor line $u_{x,\overline{x}}$ passes through the zero section of $\Vam$ which we denote by $\underline{\infty}$ and $u_{x,\overline{x}}\setminus{\underline{\infty}}\subset\Vhp_{\overline{x}}$ (as this is canonical real twistor line - see Section \ref{Feix}).

\begin{lemma}\label{aff}
For any $\overline{x}\in S^{0,1}$ and any $z\in u_{x,\overline{x}}\setminus{\underline{\infty}}$, there exists a unique Legendrian submanifold contained in $\Vhp_{\overline{x}}$. Moreover the Legendrian submanifolds with this property are affine hyperplanes. 
\end{lemma}
\begin{proof}
Inside the fibre $\Vhp_{\overline{x}}$ there is a cone $\mathcal{C}_{\overline{x}}$ given by the union of the canonical twistor lines  which are images of fibres of the Lagrangian submanifold $S^{1,0}_{\overline{x}}\subset S^{\mathbb{C}}$ (see Definition \ref{canonical}). In other words, $S^{1,0}_{\overline{x}}$ is the leaf  of the $(1,0)$-foliation through $\overline{x}$. As the curvature of the connection $D^g$ is of type $(1,1)$ on $S\subset M$, the complexified connection is flat along $S^{0,1}_x$. Hence it is flat on the preimage  of $\mathcal{C}_{\overline{x}}$ in the incident space $\cF$ (see Definition \ref{inc}) and  the Legendrians $\cW_z$ for $z\in  u_{x,\overline{x}}\setminus\{\underline{\infty},\underline{0}\}$ are the images of the parallel sections (the parallel section corresponding to $z\in\underline{0}$ is contracted to a point).	

Now consider the limit of the tangent spaces $W_z$ to $\cW_z$  along $ u_{x,\overline{x}}$ as $z\rightarrow 0\in \Vhp_{\overline{x}}$. We have that $W_z$ are Legendrian subspaces equal to $\ker (\theta)_x\cap\Vhp_{\overline{x}}$, hence  the limit is a well defined Legendrian subspace $W$ equal to $\ker (\theta)_0\cap\Vhp_{\overline{x}}$ - the dimension of $W$ is equal to $n$ as along  $ u_{x,\overline{x}}$ the space $\Vhp_{\overline{x}}$ is transversal to the contact distribution. 

	As the contact distribution on $Z$ is $\mathbb{C}^{\mult}$-invariant and the action preserves the fibres, we deduce, that the family of submanifolds $\cW_z$ for $z\in  u_{x,\overline{x}}\setminus\{\underline{\infty},\underline{0}\}$ is $\mathbb{C}^{\mult}$ invariant.
	  Hence the spaces  $W_z$ tangent to the Legendrians in $\Vhp_{\overline{x}}$ are constant along the orbit of the $\mathbb{C}^{\mult}$-action (note that $\Vhp_{\overline{x}}$ is a vector space so the tangent bundle is trivial) and equal to the limit as $z\rightarrow 0\in \Vhp_{\overline{x}}$ which is $W$. Hence the  Legendrians $\cW_z$ for  $z\in  u_{x,\overline{x}}\setminus\{\underline{\infty},\underline{0}\}$ are affine hyperplanes modeled on the vector space $W$. As the vector subspace $W$ of $\Vhp_{\overline{x}}$  is the limit of Legendrian submanifolds $\cW_z$ as $z\rightarrow 0\in \Vhp_{\overline{x}}$,  it is also Legendrian submanifold.
	
\end{proof}

\begin{remark}\label{foli}
Note that in the above Lemma we have shown that for any  $\overline{x}\in S^{0,1}$ the space $Z\cap\Vhp_{\overline{x}}$ is foliated by the Legendrians.
Moreover as rewriting $\ker (\theta)=D^g$ we have that $D^g\cap \Vhp_{\overline{x}}|_0=W$.

 We can obtain an analogous result for the space $Z\cap\Vam_x$, where $x\in S^{1,0}$.
\end{remark}

In Lemma \ref{div} we have shown that $D^g$ on $Z$ is tangent to $I|_S$ and clearly it cannot be tangent to the whole $\cD^{1,0}$ (as it is contact). However the following is true.
\begin{prop}\label{mtc}
For any $\overline{x}\in S^{0,1}$, the fibre $(\cD^{1,0}_{\overline{x}},I)$ is tangent to the contact distribution $D^g$, hence it is metrically totally complex submanifold, and in particular K\"ahler.
\end{prop}
\begin{proof}
We aim to show that  $\cD^{1,0}_{\overline{x}}=W$. Both submanifolds are vector spaces of the same dimension.
On the twistor space $Z$ the submanifold $\cD^{1,0}$ is defined as one of the two components of the zero set of the twistorial moment map. The moment map is equal to $\theta (\hat{X})$, where $\hat{X}$ is a generator of the $\mathbb{C}^{\mult}$-action on $Z$. Hence, on $\cD^{1,0}$ we have that $\hat{X}\in \ker \theta$ belongs to $D^g$ and this condition characterises $\cD^{1,0}$ in some neighbourhood of $0$.   Let $l\subset W$ be a vector line. Then for any point $w\in l\setminus\{0\} $, the vector $\hat{X}_w$ is tangent to $l$ (as this is the generator of the scalar multiplication). Hence $\hat{X}_w\in\ker(\theta_w)$ as it is tangent to a Legendrian submanifold. Therefore, $l\setminus\{0\}$  (and hence $l$) is contained in $\cD^{1,0}_{\overline{x}}$ which completes the proof.

\end{proof}

\begin{thm}\label{necessary}
Let $M$ be a  quaternion-K\"ahler $4n$-manifold with an $S^1$-isometry which has a component of the fixed point set $S$ of dimension $2n$ and such that the action on $S$ has no triholomorphic points.  Then locally near $S$, the twistor space $Z$ of $M$ can be obtained by the quaternionic Feix--Kaledin construction, the submanifold used in the construction is K\"ahler and the connection on the line bundle is unitary and has curvature proportional to the K\"ahler form.
\end{thm}
\begin{proof} 
As discussed in Section \ref{Feix} the twistor space $Z$ in this case is obtained by gluing open subsets of two vector bundles $\Vhp$ and $\Vam$. The bundle $\Vhp$ is the bundle over the leaf space of the $(1,0)$-foliation such that the fibre over a leaf is the space dual to the vector space of `affine sections' of $\cL^{1,0}\otimes (\cL^{0,1})^*$ (see Section \ref{Feix}).  By Lemma \ref{div}, for any $\overline{x}\in S^{0,1}$ the manifold $\cD^{1,0}_{\overline{x}}$ is  a vector subspace of $\Vhp_{\overline{x}}$, hence we can find an affine section $f_{\overline{x}}$ along the leaf $\overline{x}$ of the $(1,0)$-foliation such that is given by the equation $s(f_{\overline{x}})=0$, for $s\in \Vhp_{\overline{x}}$.

 Moreover, by Lemma \ref{aff} and Proposition \ref{mtc}, the Legendrians contained in $\Vhp_{\overline{x}}$ are affine subspaces modeled on $\cD^{1,0}_{\overline{x}}$, hence the leaves of the foliation by the Legendrians are given by the equations $s(f_{\overline{x}})=a$ for $a\in\mathbb{C}$. As the map $\phi_{1,0}$ is the evaluation map, the inverse images of the Legendrians in $(\cL^{1,0})^*\otimes \cL^{0,1}$ are $[af_{\overline{x}}]^{-1}$. An analogous argument gives a foliation by Legendrians  of the fibers of the bundle $\Vam$ defining corresponding sections $\tilde{f}_x$ over the leaves of the $(0,1)$-foliation. Hence for points from the intersection of  $\Vam\cap \Vhp\subset Z$ (which coincide with the images of the maps $\phi_{1,0}$ and $\phi_{0,1}$ on the complement of the zero sections) through any point we constructed two Lagrangians which are transversal hence the tangent spaces to the Lagrangians span the contact distribution. 

The pull-back by $\phi_{1,0}$ of the contact distribution is the distribution on  $(\cL^{1,0})^*\otimes \cL^{0,1}$ spanned by the tangent spaces to  $[af_{\overline{x}}]^{-1}$ and $b\tilde{f}_x$. The contact distribution on $Z$ preserves the real structure which in qFK is induced on $Z$ by the complexification. Therefore, sections  $[af_{\overline{x}}]^{-1}$ and $b\tilde{f}_x$ intersecting in a fibre of $(\cL^{1,0})^*\otimes \cL^{0,1}$ over the real submanifold $S\subset S^{\mathbb{C}}$ are conjugated. Hence the obtained distribution is defined by a complexification of a connection on the bundle $(\cL^{1,0})^*\otimes \cL^{0,1}$ over $S$ with curvature of type $(1,1)$. We have shown that the parallel sections along the leaves are inverses of non-vanishing affine sections along the leaves, hence the connection on $(\cL^{1,0})^*\otimes \cL^{0,1}$ that they generate come from the connection $\nabla$ on $\cL\otimes\mathcal{O}(1)$ used in the construction.

 By \cite{Alex2} (see Equation 2.2 and remark after Theorem 1.12), along any metrically totally complex submanifold with the complex structure $I$ of a quaternion-K\"ahler manifold $(M,g)$, the circle bundle of almost complex structures orthogonal to $I$ is $D^g$ invariant and there exists a $1$-form $\gamma$ on $M$ such that along this submanifold we have \begin{equation}\label{Dg}
  D^gJ=\gamma K,\ \  D^gK=-\gamma J,
 \end{equation}
 where $J,K$ are orthogonal. Moreover, in this setting the curvature of the connection on the circle bundle is equal to $d\gamma$ and is proportional to $g(I\cdot,\cdot)$.

Recall that the manifold $M$ admits a unique complex structure invariant under the $S^1$-action, hence globally on $M$ we have a well defined circle bundle of almost complex structures orthogonal to $I$. This bundle is a smooth submanifold of $Z$ and by Equation \ref{Dg} restricted to $S\subset M$ is preserved by the contact distribution (but globally this does not hold). Complexifying, we obtain a circle bundle over $S^{\mathbb{C}}\subset M^{\mathbb{C}}$ on the incidence space (see Definition \ref{inc}), which is preserved by the complexified connection. Taking the projection to $Z$ from the incidence space we obtain a set $\cU$ and it is easy to see using the flatness of the connection along the leaves of the $(1,0)$ and $(0,1)$-foliation that this is the union of Legendrians contained in the fibres of $\Vhp$ through the points on the canonical twistor lines belonging to the circle bundle. Note that we can analogously describe the set $\cU$ using Legendrians in fibres of bundle $\Vam$ and that this implies that the both approaches give the same set. 

The inverse image of $\cU$ is a circle bundle over $S^\mathbb{C}$ preserved by the connection which defines a complexified unitary structure on $(\cL^{1,0})^*\otimes \cL^{0,1}$. This proves that the connection $\nabla$ on $\cL\otimes\mathcal{O}(1)$ is unitary (note that the induced connection on  $(\cL^{1,0})^*\otimes \cL^{0,1}$ restricted to the real submanifold $S$ is always unitary hence to conclude that $\nabla$ itself is unitary we needed a complexified unitary structure).

The condition on the curvature follows from the fact that the curvature of the circle bundle over any metrically totally complex submanifold is proportional to $g(I\cdot,\cdot)$ hence along $S$ it is proportional to the  K\"ahler form.
\end{proof}

\begin{remark}
We have shown that the complex structure induced by the quaternion-K\"ahler moment map, in quaternionic Feix--Kaledin construction corresponds to subbundles of the bundles $\Vhp$ and $\Vam$ defined as annihilators of the parallel sections for $\nabla^c$ along the leaves of the $(1,0)$ and $(0,1)$-foliations on $S^c$ respectively.  It is straight forward to see that through quaternionic Feix--Kaledin construction, we always obtain in this way a complex structure, provided that the c-projective structure is generated by a real-analytic connection with curvature of type $(1,1)$: in this case any such a connection in the c-projective class equip $M$ with a complex quaternionic structure. Such manifolds have been studied recently by N. Hitchin \cite{Hit3}.
\end{remark}

\section{Construction of a holomorphic contact distribution on the twistor space}\label{dist}
The aim of this section is to prove that the necessary conditions for the twistor space obtained by qFK to be a twistor space of  quaternion-K\"ahler manifold are sufficient: assuming that the initial data in the construction satisfy the necessary conditions given in Theorem \ref{necessary}, we will explicitly construct a holomorphic distribution on the twistor space which is invariant under the real structure and the $\mathbb{C}^{\mult}$-action and transversal to canonical real twistor lines.
\subsection{Standardization}\label{S}
The method in \cite{Feix} involves a local identification of each of the bundles of affine functions (i.e., halves of the twistor space) with a standard flat model (note that this identification can be done for both bundles, but the identifications are not compatible with each other unless the structure is flat). We will show a similar result for qFK in the case when the c-projective structure is generated by a real-analytic connection with type $(1,1)$ curvature.

Let $(S,D)$ be a manifold with a connection with curvature of type $(1,1)$ and $\nabla_{\cL}$ be a connection compatible with the holomorphic structure with type $(1,1)$ curvature on a holomorphic line bundle $\cL$. By $\nabla$ we denote the  tensor product connection on $\cL\tp\mathcal{O}(1)$, where the connection on $\mathcal{O}(1)$ is induced by $D$. Consider the quaternionic Feix--Kaledin construction with $(S,[D]_c,\cL,\nabla_{\cL})$, where $[D]_c$ is the c-projective structure generated by $D$.
\begin{lemma}\label{stand}
The triples $((\Lp)^*\tp\Lm,\php,\Vhp)$ and $((\Lm)^*\tp\Lp,\phm,\Vam)$ can be locally standardized such that they are holomorphically equivalent to the standardized halves of the hypercomplex Feix construction.
\end{lemma}
\begin{proof}
We will prove this for $((\Lp)^*\tp\Lm,\php,\Vhp)$, as the proof for the other half is analogous. As explained by Feix (see \cite{Feix}, Lemma 2) we can pick local holomorphic coordinates $p_1,\ldots,p_n,q_1,\ldots q_n$ of $S^c$ such that the leaves of the $(1,0)$-foliation are defined setting $q_1,\ldots q_n$ to be constants (in particular we can take  $(q_1,\ldots q_n)=(\tilde{z}_1,\ldots,\tilde{z}_n)$)  and  $p_1,\ldots,p_n$ are affine functions along the leaves. Choose a section $f_0$ of $((\Lp)^*\tp\Lm)$ which is parallel along the leaves of the $(1,0)$-foliation. Then the affine sections of $\Lp\tp(\Lm)^*$ along the leaves are generated by $f_0^{-1},f_0^{-1}p_1,\ldots,f_0^{-1}p_n$ and hence  $(q_1,\ldots q_n,f_0^{-1},f_0^{-1}p_1,\ldots,f_0^{-1}p_n)$ are local coordinates on $\Vhp$. Let us trivialize  $((\Lp)^*\tp\Lm)$  using $f_0$, so that it becomes the trivial bundle $S^c\times \mathbb{C}$ with the fibre coordinate $t$. Then it is immediate to check that the map $\php$ in these coordinates is given by the formula  $$(p_1,\ldots,p_n,q_1,\ldots q_n,t)\mapsto  (t,tp_1,\ldots,tp_n,q_1,\ldots q_n).$$
\end{proof}
\subsection{Construction of the contact distribution}\label{2}
First we discuss the formula of the connection in local coordinates. Let $\cL\tp\mathcal{O}(1)$ be a line bundle with a holomorphic connection whose curvature is a scalar multiple $c$ of the K\"ahler form. Then in a local trivialization the connection takes the form $$\nabla=d+c\sum_{i=1}^n a_i(z,\overline{z}) dz_i,$$ where 
$d\Sigma_{i=1}^n a_idz_i=\Sigma_{i=1}^n da_i\wedge dz_i=\omega$.
The condition that $\nabla$ is unitary means that there exists positive valued function $h(z,\overline{z})$ such that $$c\sum_{i=1}^n a_i dz_i=\partial \log h=\sum_{i=1}^n\frac{\partial \log h}{\partial z_i}dz_i.$$ In this setting $\eta:=\log h$ is a K\"ahler potential for $\omega$ as $d\partial=\overline{\partial}\partial$.

Now consider the complexified picture. The corresponding connection\\ $d+c\Sigma_{i=1}^n a_i(z,\tilde{z}) dz_i$ on $\Lp$ is flat along the leaves of the $(1,0)$-foliation and the function $e^{-\eta (z,\tilde{z})}=[h(z,\tilde{z})]^{-1}$ is parallel along the leaves of the $(1,0)$-foliation. Observe also that the corresponding connection on $\Lm$ is  $d+c\Sigma_{i=1}^n \tilde{a}_i (z,\tilde{z})d\tilde{z}_i$, where $\tilde{a}_i(z,\tilde{z}):=a_i(\tilde{z},{z})$ is a complexification of $a_i(\overline{z},z)$, hence the connection on $\Lp\tp[\Lm]^*$ is given in this trivialization by the formula:
$$d+c\sum_{i=1}^n a_i dz_i-c\sum_{i=1}^n \tilde{a}_id\tilde{z}_i.$$
\begin{lemma}\label{distr}
The  holomorphic distribution given by a complexification of the connection $\nabla$ on $\mathbb{P}(([\cL\tp\mathcal{O}(1)]^{*})^{\mathbb{C} })=\mathbb{P}(\Lp^{*}\ds\Lm^{*})$ extends via the map $\phi$ to the holomorphic distribution on $Z$.
\end{lemma}
\begin{proof}
We will show that the distribution extends to $Z\cap \Vhp$. The proof for the other half is analogous. The connection distribution on $[\Lp]^*\tp\Lm$ is given as the kernel of the following $1$-form
$$\varphi:= df-fc\Sigma_{i=1}^n a_i dz_i+cf\Sigma_{i=1}^n \tilde{a}_id\tilde{z}_i,$$
where $a_i,\tilde{a}_i$ are defined above and $f$ is a fibre coordinate.  Now observe that, as $(1,0)$-foliation is by definition given by $\tilde{z}=const$ and 
 $d \tilde{a}_i\wedge d\tilde{z}_i=-\omega$ we have that $\tilde{a}_i$ are affine along the leaves (see \cite{Feix}) and hence the standardizing coordinates $p_i$ from Lemma \ref{stand} can be taken to be equal $\tilde{a}_i$.
Moreover we can set the standardizing trivialization $f_0$ to be equal $e^{\eta}$. By setting $t:=e^{-\eta}f$ we pass to the standardizing trivialization of $\Lp\tp[\Lm]^*$  and the connection $\nabla$ in this coordinates is equal
$$d[e^{\eta}t]- ce^{\eta}t\Sigma_{i=1}^n a_i dz_i+ce^{\eta}t\Sigma_{i=1}^n \tilde{a}_id\tilde{z}_i=e^{\eta}[dt+ct\Sigma_{i=1}^n \frac{\partial \eta}{\partial\tilde{z}_i}d\tilde{z}_i+ct\Sigma_{i=1}^n \tilde{a}_id\tilde{z}_i].$$
As $h$ on $S$ is a positive function, its logarithm $\eta$ is real valued on $S$ hence $[\partial \eta](z,\overline{z})=[\overline{\partial} \eta] (\overline{z}, z)$ and therefore $\frac{\partial\eta}{\partial\tilde{z}_i}=\tilde{a}_i$.
Using the standardized coordinates on $\Vhp$ we get that the form $e^{-\eta}\varphi$ extends to $\Vhp$ which completes the proof.

\end{proof}
\begin{remark}
The contact form on a twistor space has values in a line bundle which after restriction to any of the real twistor lines is isomorphic to $\mathcal{O}(2)$. In our setting this is represented by the fact that although the distribution given by $\varphi$ extends to $Z$, the form $\varphi$ does not. Indeed using the notation from the proof of Lemma \ref{distr}, the form $e^{-\eta}\varphi$ extends to one half of the twistor space (namely to $\Vhp$) while  $e^{\eta}\varphi$ extends to the other one ($\Vam$).
\end{remark}
\begin{thm}\label{co}
Let $S$ be a real-analytic K\"ahler manifold, $\cL$ a holomorphic line bundle on $S$ with a connection $\nabla_{\cL}$ such that the connection $\nabla$ induced on $\cL\otimes\mathcal{O}(1)$ is a unitary connection with curvature equal to $c\omega$, where $\omega$ is the K\"ahler form. Then the twistor space $Z$ obtained by qFK from $(S,\omega, \cL,\nabla_{\cL})$ is a holomorphic contact manifold which is a twistor space of a quaternion-K\"ahler manifold $M$ with scalar curvature equal to $2c$. Moreover the  $S^1$-action on $M$ coming from the construction is an isometry.

\end{thm}
\begin{proof}
By Lemma \ref{distr}, we have already proved that $Z$ is a holomorphic contact manifold. By construction, the obtained distribution is preserved by the real structure on $Z$ and transversal to canonical real twistor lines. 
The existence of such contact distribution on the twistor space is equivalent to the existence of a pseudo quaternion-K\"ahler metric. As $S$ is metrically totally complex, we have that for $J$ anti-commuting with $I$ the tangent bundle splits $TM=TS\ds JTS$ and this splitting is $g$-orthogonal. Moreover, we have also that $g(J\cdot,J\cdot)=g(\cdot,\cdot)$ hence the positive-definiteness of the metric follows directly from the positive-definiteness of the initial K\"ahler metric.
In the proof of Theorem \ref{necessary} we have shown that $2c\omega$ is the curvature of the circle bundle of orthogonal complex structures over the metrically totally complex manifold $S\subset M$. By \cite{Alex2}, this implies that $2c$ is the scalar curvature of the underlying quaternion-K\"ahler manifold.

Finally, it is straightforward to check that the $C^{\mult}$-action on $Z$ given by scalar multiplication in fibres preserves the contact distribution hence it corresponds to an isometry on $M$.
\end{proof}
\begin{remark}
Similarly as in the hyperk\"ahler case of Feix, we can generalize this argument to pseudo quaternion-K\"ahler metric - by the construction we obtain quaternion-K\"ahler metrics with signature $(2p,2q)$ when starting from a K\"ahler metrics with signature $(p,q)$.
\end{remark}

\section{Gluing along the submanifold}\label{gl}

Firstly note that the qFK construction  (\cite{BC}) is global in some neighbourhood of $S$: although during the construction we have to restrict several times to local neighbourhoods, they are always tubular neighbourhoods of the whole manifold $S$ (and on the level of twistor space it is a tubular neighbourhood of the union of canonical real twistor lines - i.e., the one that correspond to $S$). Hence in particular if $S$ is compact  K\"ahler manifold and $\cL$ is a holomorphic line bundle on $S$ with a real analytic holomorphic connection $\nabla$ with curvature of type $(1,1)$ then we can obtain a quaternionic manifold $M$ which contains the whole $S$ as the zero set of an $S^1$-symmetry and any two quaternionic manifolds which on the fixed points set $S$ induce the same structure (i.e., the  K\"ahler metric and the line bundle with a connection) are locally equivalent near $S$. Also observe, that the construction of the contact distribution in Section \ref{2} glue uniquely along $S$ (viewed as the zero section in $\Vhp$ and $\Vam$ respectively), hence we obtain the following proposition.
\begin{prop}
Suppose that $S$ is a compact K\"ahler manifold with a holomorphic line bundle $(\cL\otimes\mathcal{O}(1), \nabla)$ with a connection satisfying the assumptions of Theorem \ref{co}. Then there exist (not necessarily compact) quaternion-K\"ahler manifold $M$ containing $S$ as a component of the fixed points set of an isometric $S^1$-action with no triholomorphic points on $S$. Moreover all such $M$ are equivariantly isomorphic near $S$.
\end{prop}
However, we can prove a stronger fact. Note that as the contact line bundle (i.e. the bundle in which the contact form  takes the values) is the vertical bundle of the twistor fibration, along the zero section of $\Vhp$ it is isomorphic to $\Lp\tp[\Lm]^*|_{S}$, where $S$ is the diagonal in $S^{\mathbb{C}}$. By definition of the complexification we have that along the diagonal $\overline{\Lp}=\Lm$ and, as the bundle  $\cL\otimes\mathcal{O}(1)$ is unitary (by Theorem \ref{necessary}), we also have that $\Lp=\overline{\Lp}^*$. Hence, we conclude that the contact line bundle along the zero section is equal to $\Lp^{\otimes 2}$, which means that if $S$ is a fixed point set component of a quaternion-K\"ahler manifold with an isometric $S^1$-action with no triholomorphic points on $S$, then the bundle $\cL\otimes\mathcal{O}(1)$ used in the construction must be a square root of a holomorphic bundle with a connection which is globally defined  on $S$. Now we will show that the converse is also true.
\begin{thm}\label{3}
 Suppose that $(S,\omega)$ is a K\"ahler manifold equipped with a holomorphic line bundle $\mathfrak{L}$  with a unitary connection on $S$ which has curvature $C\omega$ proportional to the K\"ahler form. Let $\cL$ be such that $\cL\otimes\mathcal{O}(1)$ is a local square root of $\mathfrak{L}$ and $\nabla$ is the induced connection on $\cL$. Then there exists a quaternion-K\"ahler manifold $M$ with scalar curvature $C$ equipped with an isometric $S^1$-action such that $S$ is a component of the fixed point set and $M$ locally arises by the quaternionic Feix--Kaledin construction from  $\cL$ with the induced connection. Moreover such a manifold is unique near $S$.
 \end{thm}

\begin{proof}
By Theorem \ref{2} such $M$ exists locally near any open subset of $S$ where the bundle $\cL$ is defined.

Firstly observe that if $\cL$ is a unitary line bundle then the complexified hermitian product on $\cL^{\mathbb{C}}=\cL^{1,0}\oplus\cL^{0,1}$ over $S^c$ gives a paring between $\cL^{1,0}$ and $\cL^{0,1}$ hence the bundle $\cL^{1,0}\otimes\cL^{0,1}$ is trivial and as a consequence the bundles $(\Lp)^{\tp k}\tp(\Lm)^{\tp k}$ are canonically trivial for any $k$ rational. Hence the bundle $(\Lp)^*\tp \Lm$ is canonically isomorphic with $(\Lm)^2$. Therefore we obtain a bundle $\cF$ of projective lines over an open neighbourhood of the diagonal $S$ in $S^{\mathbb{C}}$ by gluing globally the vector bundles $\Lm^2$ and $\Lp^2$  using canonical triviality of $\Lp\tp\Lm$ in the following way:
$$\Lm^2\cong[(\Lp)^*\tp \Lm]\tp[\Lp\tp\Lm]\cong [(\Lm)^*\tp\Lp]^*\tp[\Lp\tp\Lm]^*\cong(\Lp^2)^*.$$
 In this setting locally $\cF$ is equal to $\mathbb{P}(\Lp^{*}\ds\Lm^{*})$. 
Denote by $\Vhp_e$ the dual to the bundle of affine sections of  $(\Lm^2)^*$  along the leaves of the $(1,0)$-foliation and by $\Vam_e$ the dual to the bundle of affine sections of $(\Lp^2)^*$ along the leaves of the $(0,1)$-foliation. Then we obtain the following diagram:
\begin{equation} \label{dig2}
\xymatrix@R=1.1cm@C=.9cm{
& (\Lm)^2\ar@{^{(}->}[dl] \ar[r]^-{\phi_{1,0}^e} & \Vhp_e \ar@{~>}[dr] &\\
\cF\ar@{.>}[rrr]^{\phi^e}&&&Z^e\\
&(\Lp)^2\ar@{_{(}->}[ul] \ar[r]^-{\phi_{0,1}^e}&\Vam_e\ar@{~>}[ur]&
}
\end{equation}
The bundle $\Vhp_e$ is defined globally over $S^{0,1}$ and using again canonical triviality of $\Lp\tp\Lm$ we get that it is canonically isomorphic to $\Vhp$ over any subset of $S^{1,0}$ where $\Vhp$ is defined. Similarly $\Vam_e$ is an extension of $\Vam$ to the bundle globally defined over $S^{0,1}$.
Therefore the Diagram \ref{dig2} is an extension of the Diagram \ref{dig1} from Section \ref{Feix} summarizing qFK. Hence the obtained manifold $Z^e$ is locally a twistor space arising by qFK from $(\cL,\nabla)$ and by construction and by Theorem \ref{2} it is immediate to see that it is a twistor space of a quaternion-K\"ahler manifold with the required properties which contains the whole manifold $S$.

The uniqueness of $M$ near $S$ follows from Theorem \ref{necessary} and from \cite{BC} (Theorem $4$): Suppose that $M$ is a quaternionic manifold with the required properies containing the whole manifold $S$. Then the incidence space of the twistor fibration restricted to the canonical twistor lines is locally near the canonical real twistor lines a bundle of projective lines over a complexification of $S$ with two distinguished disjoint sections corresponding to the opposite complex structures. Moreover, the contact distribution defines on the affine parts of the bundle a complexification of the connection used locally in the construction. In particular after the restriction of the bundle to the diagonal we obtain the connection used in the construction on $(\Lp)^*\tp\Lm\cong (\Lm)^2$ which is globally defined on $S$. Using local equivalence of complexifications it follows  that the affine parts of the incidence space must be locally isomorphic to the affine parts of $\cF$ which are glued together in the same way. It follows that the quaternion-K\"ahler manifolds are locally equivalent near $S$.
\end{proof}
\begin{remark}\label{bundle}
In the proof of the above theorem instead of using the triviality of $\Lp\tp\Lm$ we can use the triviality of $(\Lp\tp L^{1,0})\tp(\Lm\tp L^{0,1})$ for any $L$ unitary defined globally over $S$ - then $\cF$ will be constructed as gluing of different bundles but the resulting space as well as the bundles $\Vhp_e$ and $\Vam_e$ will remain unchanged. In particular  in the case when $S=\mathbb{CP}^n$, we can use this observation to identify the bundles $\Vhp$ and $\Vam$: the bundle $\mathcal{O}(1)$ (see Definition \ref{o}) is well defined over the whole $\mathbb{CP}^n$ and the bundle of affine sections of $\mathcal{O}(1,0)$ along the leaves of the $(1,0)$-foliation is trivial. Now the bundle of affine sections of $\Lp\tp\Lm^*$ is isomorphic to the bundle of affine sections of $\mathcal{O}(1,1)\otimes (\Lm^*)^2$ and taking the dual we get that $\Vhp$ is isomorphic to $\mathcal{O}(-1)\otimes (\Lm)^2\otimes\mathbb{C}^{2n+1}$. We will use this approach in discussion of examples in the next section.
\end{remark}
We can further summarize the obtained results in the following theorem.
\begin{thm}
Let $S$ be a compact K\"ahler manifold of complex dimension $n$ and $L$ a line bundle on $S$ such that it admits a unitary connection which curvature is proportional to the K\"ahler form. Then there exist, possibly non-compact holomorphic contact $(2n+1)$-manifold $Z$  containing $S$ as Legendrian submanifold which is a twistor space of a quaternion-K\"ahler manifold and such that $L$ is the contact line bundle restricted to $S$. Moreover $Z$ admits a holomorphic $ \mathbb{C}^{\mult}$-action fixing only $S$ and $\overline{S}$ corresponding to an isometric $S^1$-action.

Any two holomorphic contact manifolds with the above properties are biholomorphic on some neighbourhood of the set of all real twistor lines through $S$.
\end{thm}
\begin{proof}
The condition on curvature implies that the unitary connection on $L$ is Hermite-Einstein. By Kobayashi--Hitchin correspondence, Hermite-Einstein connections on any holomorphic line bundle over compact K\"ahler manifold are unique, hence the contact line bundle determines unique unitary connection  with curvature proportional to the K\"ahler form which by Theorems \ref{2}, \ref{co} and \ref{3} completes the proof.
\end{proof}

\section{Examples and applications}\label{ex}
In this section we will discuss the examples of $\mathbb{CP}^{2n+1}$ and the flag manifold $F_{1,n+1}(\mathbb{C}^{n+2})$ which already appeared in \cite{BC} and see how this fits into the setting of Remark \ref{bundle}.
\begin{example}
Let $S=\mathbb{CP}^n$ with the standard K\"ahler structure. As discussed in \cite{BC}, by taking $\cL$ trivial (which means that $\Lp=\mathcal{O}(1)$)by qFK we obtain $Z=\mathbb{CP}^{2n+1}$ which is the twistor space of the quaternionic projective space. In this case we consider affine sections of $\mathcal{O}(1,-1)$ along the leaves of the $(1,0)$-foliation, hence by Remark  \ref{bundle} it should mean that in the construction we obtain $Z$ by gluing two copies of the bundle $\mathcal{O}(1)\tp\mathbb{C}^{2n+1}$ over $\mathbb{CP}^n$ using the blow down from $\mathbb{P}(\mathcal{O}(-1,0)\oplus\mathcal{O}(0,-1))$. Indeed, fixing a projective $n$-space $S$ in $\mathbb{CP}^{2n+1}$ there is a natural embedding of the bundle $\mathcal{O}(1)\tp\mathbb{C}^{2n+1}$ into $\mathbb{CP}^{2n+1}$  such that $S$ is the zero section:
denoting points in $\mathbb{CP}^{2n+1}$ by $[x_0:\ldots :x_n:y_0:\ldots :y_n]$ we can take $S$ to be given by the equation $(y_0,\ldots,y_n)=0$ and $\overline{S}$ by $(x_0,\ldots,x_n)=0$ and then on the complement of $\overline{S}$, the space $Z$ can be identified as the total space of $\mathcal{O}(1)\tp\mathbb{C}^{2n+1}$. The blow down map from the affine part $\mathcal{O}(-1,1)$ over $S^\mathbb{C}=\mathbb{CP}^n\times\mathbb{CP}^n$ is given explicitly by sending element $l\otimes [y_0,\ldots,y_n]\in\mathcal{O}(1,-1)_{[x_0:\ldots :x_n]\times[y_0:\ldots :y_n]} $ to  $[x_0:\ldots :x_n:y_0:\ldots :y_n]$, where $l(x_0,\ldots, x_n)=1$ and we get the analogous result for the other half of the twistor space. Note that the $\mathbb{C}^{\mult}$-action on $\mathbb{CP}^{2n+1}$ given by $[\lambda x_0:\ldots :\lambda x_n:\lambda^{-1}y_0:\ldots :\lambda^{-1}y_n]$ fixes exactly submanifolds $S$ and $\overline{S}$ acts as the scalar multiplication in these fibres.

\end{example}

\begin{example}
Let $S=\mathbb{CP}^n$ with the standard K\"ahler structure and let $\cL$ be $\mathcal{O}(-\frac{1}{2})$ (hence $\Lp=\mathcal{O}(\frac{1}{2})$). As discussed in \cite{BC} in this case by qFK we obtain the flag manifold $Z=F_{1,n}(\mathbb{C}^{n+2})$ which is the twistor space of the Grassmannian $Gr_2(\mathbb{C}^{n+2})$. By Remark \ref{bundle}, $Z$ is the gluing of of two trivial rank $n+1$ bundles using the blow downs from $\mathcal{O}(-1,0)$ and $\mathcal{O}(0,-1)$ respectively and the canonical triviality of $\mathcal{O}(\frac{1}{2},\frac{1}{2})$.

 Indeed, let $\mathbb{C}^{n+2}$  with standard coordinates $(p,q_0,\ldots, q_n)$ be equipped with standard hermitian product and the $\mathbb{C}^{\mult}$-action $(\lambda^{-1}p,\lambda q_0,\ldots\lambda q_n)$ and denote by $a$ the line defined by the equation $q_0=\cdots=q_n=0$. Then the real structure on $Z$ is defined by $\perp$ and the $\mathbb{C}^{\mult}$-action on $\mathbb{C}^{n+2}$ induce a  $\mathbb{C}^{\mult}$-action on $Z$ fixing two $\perp$-conjugated submanifolds of $Z$ isomorphic to $\mathbb{CP}^n$ namely $\{(l,K)\in  F_{1,n}(\mathbb{C}^{n+2}): \ l=a\}$ and  $\{(l,K)\in  F_{1,n}(\mathbb{C}^{n+2}): \ K=a^{\perp}\}$. The $\mathbb{C}^{\mult}$-action defines fibrations and vector bundle structures on   $\{(l,K)\in  F_{1,n}(\mathbb{C}^{n+2}): \ l\cap a^{\perp}=0\}$ and $\{(l,K)\in  F_{1,n}(\mathbb{C}^{n+2}): \ K\cap a=0\}$ such that the action is the scalar multiplication in the fibres and its inverse respectively. Consider $(l,K)\in  F_{1,n}(\mathbb{C}^{n+2})$ such that $ l\cap a^{\perp}=0$ and set $k=K^{\perp}$. Let us write $l=\spanu(\tilde{P}+\tilde{Q})$ and $k=\spanu(\hat{P}+\hat{Q})$ for $\tilde{P},\hat{P}\in a$, $\tilde{Q},\hat{Q}\in a^{\perp}$ and $\tilde{P},\hat{Q}\neq 0$; without loss of generality we can fix $\tilde{P}$ to be defined by $p=1$. Then the action of an element $\lambda\in\mathbb{C}^{\mult}$ on $(l,K)$ gives $(\spanu(\lambda^{-1}\tilde{P}+\lambda\tilde{Q}),\spanu(\lambda\hat{P}+\lambda^{-1}\hat{Q})^{\perp})$ and the limit as $\lambda\rightarrow 0$ is $(a,\spanu(\hat{Q})^{\perp})$. We conclude that $\{(l,K)\in  F_{1,n}(\mathbb{C}^{n+2}): \ l\cap a^{\perp}=0\}$ is the trivial rank $n+1$ bundle over $\mathbb{P}(\mathbb{C}^{n+2}/a)\cong\mathbb{CP}(a^{\perp})$. This is because for fixed $\tilde{P},\tilde{Q}$ and a line $\spanu(\hat{Q})$ there is a unique line $\spanu(\hat{P}+\hat{Q})$, such that $\spanu(\tilde{P}+\tilde{Q})\perp \spanu(\hat{P}+\hat{Q})$, hence projecting flags to $\tilde{Q}$ gives the trivialization of the bundle. In analogous way we get that  $\{(l,K)\in  F_{1,n}(\mathbb{C}^{n+2}): \ K\cap a=0\}$ is the trivial  rank $n+1$ bundle over $\mathbb{P}(a^{\perp})$.
 
 The blow down map from $\mathcal{O}(-1,0)$ over $\mathbb{P}(a^{\perp})\times\mathbb{P}(a^{\perp})$ to  $\{(l,K)\in  F_{1,n}(\mathbb{C}^{n+2}): \ l\cap a^{\perp}=0\}$ is given explicitly by sending $(\tilde{Q},\spanu (\hat{Q}))$ to the flag  $\spanu(\tilde{P}+\tilde{Q})\subset \spanu(\hat{P}+\hat{Q})^{\perp}$ and the gluing with $\mathcal{O}(0,-1)$ is given by the hermitian product which manifest itself in the equation  $\spanu(\tilde{P}+\tilde{Q})\perp \spanu(\hat{P}+\hat{Q})$.  Recall that we can use the triviality of $\mathcal{O}(\frac{1}{2},\frac{1}{2})$ (given by complexified hermitian product) to identify $\mathcal{O}(-1,0)$ with $\mathcal{O}(-\frac{1}{2},\frac{1}{2})$ and $\mathcal{O}(0,-1)$ with $\mathcal{O}(\frac{1}{2},-\frac{1}{2})=(\mathcal{O}(-\frac{1}{2},\frac{1}{2}))^{*}$. It is straightforward to check that after this identification the gluing is $t\mapsto t^*$, as required in qFK.
\end{example}
\subsection*{C-map}
The presented construction of quaternion-K\"ahler metric extends the Hitchin twistorial description of c-map construction ( \cite{Hit2}, see also \cite{ACDM}, \cite{Hay}, \cite{MaSw} and \cite{CH} for non-metric case) by allowing the fixed points of the $S^1$-action. Indeed, let $Z$ be a quaternion-K\"ahler manifold obtained by qFK from a K\"ahler manifold $S$ and let $\cD^{1,0}$ be the divisor from Lemma \ref{div} and $\cD^{0,1}$ its conjugate (corresponding to the complex structure $-I$). Then  the $\mathbb{C}^{\mult}$ quotient of the total space of the the divisor bundle of $\cD^{1,0}-\cD^{0,1}$ over $Z$ is the hyperk\"ahler manifold obtained by standard Feix construction from $S$. This has been shown in the proof of Theroem 6 from \cite{BC} for $S$ K\"ahler-Einstein and the argument extends to our situation. As a consequence we obtain the following theorem.
\begin{thm}
Let $S$ be a K\"ahler manifold and $(\cL,\nabla)$ a holomorphic line bundle with a unitary connection whose curvature is proportional to the K\"ahler form. Let $M$ be a quaternion-K\"ahler obtained by the qFK construction form $(S,\cL,\nabla)$ and  $\hat{M}$ the hyperk\"ahler manifold obtained by the standard Feix construction from $S$. Then $M$ and $\hat{M}$ are related by the Haydys hyperk\"ahler/quaternion-K\"ahler correspondence.
\end{thm}

\acknowledge The author would like to thank Jaros\l{}aw Buczy\'nski and David Calderbank for helpful comments and Vincente Cort\'es for his remark about necessity of K\"ahler-Einstein condition. This work is supported by the National Science Center, Poland, project ``Complex contact manifolds and geometry of secants'', 2017/26/E/ST1/00231, by the grant 346300 for IMPAN from the Simons Foundation and the matching 2015-2019 Polish MNiSW fund.

\bibliographystyle{amsalpha}

\providecommand{\bysame}{\leavevmode\hbox to3em{\hrulefill}\thinspace}
\providecommand{\MR}{\relax\ifhmode\unskip\space\fi MR }
\providecommand{\MRhref}[2]{%
 \href{http://www.ams.org/mathscinet-getitem?mr=#1}{#2}
}
\providecommand{\href}[2]{#2}

\end{document}